\documentclass[psamsfonts]{amsart}
\pdfoutput=1

\usepackage{amssymb,amsfonts,amssymb,amsthm,mathabx}
\usepackage[all,arc]{xy}
\usepackage{enumerate}
\usepackage{mathrsfs}
\usepackage{hyperref}
\usepackage{dsfont}

\usepackage{bm}
\usepackage{pgfplots}
\usepackage[T1]{fontenc}
\allowdisplaybreaks

\usepackage{amsrefs}

\newtheorem{theorem}{Theorem}

\newtheorem{lemma}{Lemma}

\newtheorem*{remark}{Remark}

\newtheorem{definition}{Definition}

\newcommand{\R}{\mathbb{R}}

\newcommand*{\bigchi}{\mbox{\large$\chi$}}

\newcommand{\la}{\langle}
\newcommand{\ra}{\rangle}
\newcommand{\Lp}{L_p}
\newcommand{\Lone}{L_1}

\newcommand{\Linfty}{L_\infty}

\newcommand{\Ltwo}{L_2}

\newcommand{\Lpprime}{L_{p^\prime}}
\newcommand{\Ltwozero}{L_2^\circ}

\newcommand{\Ltwoc}{L_2^c}
\newcommand{\Ltwor}{L_2^r}

\newcommand{\Ltwoczero}{L^{\circ,c}_2}

\newcommand{\Hc}{\mathcal{H}^c}

\newcommand{\Hone}{\mathrm{H}_1}
\newcommand{\Honec}{\mathrm{H}_1^c}
\newcommand{\Honer}{\mathrm{H}_1^r}
\newcommand{\Hardyc}{\Honec(\mathcal{A})}
\newcommand{\Hardyr}{\Honer(\mathcal{A})}
\newcommand{\Hardy}{\Hone(\mathcal{A})}

\newcommand{\BMO}{\mathsf{BMO}}

\newcommand{\MMM}{\mathcal{M}}

\newcommand{\BHH}{B(H)}

\title{Molecules and Calder\'on-Zygmund operators with noncommuting kernels on $\Honec$}

\author[A.I. Cano-M\'armol]{Antonio Ismael Cano-M\'armol}
\address{Department of Mathematics \\ Baylor University \\
1301 S University Parks Dr\\
Waco, TX 76798, USA.}
\email{AntonioIsmael\_CanoMa@baylor.edu}

\begin{document}


    
    \maketitle


    \begin{abstract}
        We study the description of semicommutative Hardy spaces in terms of molecules. We use this characterization to obtain $\Honec-\Honec$ estimates for Calderón-Zygmund operators with kernels with values in a semifinite von Neumann algebra $\MMM$.
    \end{abstract}

    \section*{Introduction}

    In this paper, we introduce sufficient conditions for the boundedness of Calder\'on-Zygmund operators with noncommuting kernels from the operator-valued version of the Hardy space $\Hone$ to itself. This complements the results which were obtained in the work by the author and Ricard \cite{CR24}, and can be framed within the theory of semicommutative Calder\'on-Zygmund operators. Let $(\mathcal{M},\tau)$ be a semifinite von Neumann algebra of operators on a separable Hilbert space, equipped with a normal semifinite faithful trace $\tau$. Denote by $\mathcal{A}$ the weak operator closure of the space of essentially bounded (strongly measurable) functions $f : \mathbb{R} \longrightarrow \mathcal{M}$ acting on $L_2(\mathbb R ; L_2(\mathcal M))$. The von Neumann algebra $\mathcal{A}$ can be identified with the tensor product $\Linfty(\mathbb{R}) \overline{\otimes} \mathcal{M}$ equipped with the trace
\begin{align*}
    \varphi(f) = \int_{\mathbb{R}} \tau(f(x)) \ dx.
\end{align*}  
We will restrict ourselves to dimension 1, even though our arguments extend trivially to any finite dimension namely for $\Linfty(\mathbb R^n) \overline \otimes \mathcal M$.

The noncommutative $\Lp$-spaces associated with $\mathcal{A}$ are indeed vector-valued $\Lp$-spaces: more clearly \cite[Chapter 3]{P98}
$$
\Lp(\mathcal{A}) = \Lp(\mathbb{R};\Lp(\mathcal{M})),
$$    
for $1\leq p<\infty$. However, in this note we will discuss the boundedness of operators on the Hardy space associated to $\mathcal{A}$. More clearly, boundedness results of the type $\Hone \longrightarrow \Hone$. This question was studied for scalar-valued functions \cite{CW77,Mey90,MeyC97} as well as for the vector-valued setting \cite{F90,H06}, where the
existence of the atomic decomposition plays an essential role. This technique does not seem to have been exploited as often in the noncommutative setting except perhaps in \cite{HLMP14} and more recently in \cite{CR24}. Mei \cite{M07} was the first to introduce the so-called \emph{operator-valued Hardy space} $\Hone(\mathbb{R},\mathcal{M})$ in this context via noncommutative equivalents of the Poisson integral, the Lusin area integral and the Littlewood-Paley $g$ function. These techniques allowed Mei to identify the dual space of $\Hone(\mathbb{R},\mathcal{M})$, which is denoted by $\BMO(\mathbb{R},\mathcal{M})$, in the spirit of the classical argument by Fefferman and Stein \cite{FS72}.

    More recently, the author and Ricard \cite{CR24} introduced an alternative definition of the operator-valued Hardy space via a ``new'' atomic decomposition of the Hardy space. A \emph{$c$-atom} is a function $a \in \Lone(\mathcal{A})$ which admits a factorization of the form $a = b h$ for some function $b : \R \rightarrow \Ltwo(\mathcal{M})$ and an norm-one operator $h \in \Ltwo(\mathcal{M})$, satisfying
        \begin{enumerate}
            \item $\mathrm{supp}_{\mathbb{R}}(b) \subseteq I$ for some interval $I$,
            \item $\displaystyle \int_{I} b = 0$,
            \item $\displaystyle \|b\|_{\Ltwo(\R;\Ltwo(\mathcal{M}))} \leq \frac{1}{\sqrt{|I|}}$.
        \end{enumerate}
        Then, the \emph{column Hardy space} $\Hardyc$ is defined to be the subspace of elements in $\Lone(\mathcal{A})$ of the form
        \begin{align*}
            \sum\limits_{i=0}^{\infty} \lambda_i a_i \ \mathrm{where \ } (\lambda_i)_i \in \ell_1 \ \mathrm{and} \ (a_i)_i \ c\mbox{-}\mathrm{atoms}
        \end{align*}
        with respect to the norm
        \begin{align*}
            \|f\|_{\Honec(\mathcal{A})} = \inf \Big\{ \sum_{i=0}^{\infty} |\lambda_i| \ : \ f = \sum_{i = 0}^{\infty} \lambda_i a_i \Big\}.
         \end{align*}
    The row space $\Hardyr$ is defined analogously via $r$-atoms of the form $a=hb$, and $\Hardy = \Hardyc + \Hardyr$. 
    
    Let $\mathcal S$ denote the set of compactly supported  essentially bounded functions $\mathbb R\to \Linfty \cap \Lone(\mathcal{M})$ (measurable with values in $L_1$). Let $T$ be a bounded operator on $\Ltwo(\mathcal{A})$ for which there exists a kernel $K : \mathbb{R} \times \mathbb{R} \ \setminus \ \{x=y\} \longrightarrow \mathcal{M}$ such that for every pair of intervals $I$, $J$ satisfying $d(I,J) > 0$, there exists $K_{I,J} \in \Linfty(I\times J; \MMM)$ such that 
    $$
    K(t) = K_{I,J}(t) \ \mbox{for any } t \in I \times J
    $$
    and
    \begin{align*}
        \int T(f)(x) g(x) \ dx = \int \int K_{I,J}(x,y) f(y) g(x) \ dx \ dy
    \end{align*}
    holds for any $f,g \in \mathcal S$ satisfying $\mathrm{supp}\|f\|_{\Ltwo(\mathcal{M})} \subset J$ and $\mathrm{supp}\|g\|_{\Ltwo(\mathcal{M})} \subset I$. More technical details on this definition can be found in \cite{CR24}. Under these assumptions, we say that $T$ is a \emph{Calder\'on-Zygmund operator with kernel $K$}. Moreover, if $T$ fulfills a right-modularity condition, that is,
    \begin{align*}
        T(fh) = T(f)h
    \end{align*}
    for any $f \in \Ltwo(\mathcal{A})$ with compact support and $h \in \mathcal{M}$, we say that $T$ is a \emph{left Calder\'on-Zygmund operator}. The main result in \cite{CR24} states that whenever the kernel $K$ satisfies the \emph{Hörmander condition}, that is, that for some $\lambda>0$ there holds
    \begin{align}\label{eq:hormanderCondition}
        \int_{|x-y| \geq \lambda |y^\prime -y|} \|K(x,y)-K(x,y^\prime)\|_{\mathcal{M}} \ dx < \infty,
    \end{align}        
    then $T$ is bounded from $\Honec(\mathcal{A})$ to $\Lone(\mathcal{A})$. As a straightforward consequence, there follows that the Hardy space $\Hone(\mathcal{A})$ coincides with the one introduced by Mei \cite{M07}.

      Nonetheless, the condition \eqref{eq:hormanderCondition} is not sufficient to prove the boundedness of Calderón-Zygmund operators from $\Hardyc$ to itself. Instead, a stronger assumption is required, namely the \textit{Lipschitz condition}: there exists some $\lambda > 0$ and $\gamma \in (1/2,1]$ such that
    \begin{align}\label{eq:smoothnessCondition}
        \|K(x,y)-K(x,y^\prime)\|_{\mathcal{M}} \lesssim \frac{|y^\prime -y|^\gamma}{|x-y|^{1+\gamma}} \mbox{ whenever } |y^\prime - y| \leq \frac{|x-y|}{\lambda}.
    \end{align}

     \begin{theorem}\label{thm:CZH1c}
        Let $\mathcal{M}$ be a von Neumann algebra. Let $T$ be a left Calder\'on-Zygmund operator with associated kernel $K : \mathbb{R} \times \mathbb{R} \ \setminus \ \{x = y\} \longrightarrow \mathcal{M}$. If $K$ satisfies the Lipschitz condition \eqref{eq:smoothnessCondition} and $\int_{\R} T(b) = 0$ for every $c$-atom $a=bh$, then $T$ extends to a bounded operator from $\Hardyc$ to $\Hardyc$.
    \end{theorem}
    
    This result heavily relies on the connection between the theory of vector-valued Hardy spaces and the semicommutative Hardy space $\Honec(\mathcal{A})$. More clearly, it is based on the decomposition of $\Honec(\mathcal{A})$ into column-valued versions of \emph{molecules}, which have been widely studied in the classical setting \cite{CW77, Mey90, MeyC97}. Analogous statements follow for right-Calder\'on-Zygmund operators on $\Honer(\mathcal{A})$ and the vector-valued setting. Therefore, operators with scalar-valued kernels satisfying both modularity conditions happen to be bounded on the full Hardy space $\Hone(\mathcal{A})$.

    \section{Preliminaries}

    \subsection*{Noncommutative spaces $\Lp(\mathcal{M};\Ltwoc(\Omega))$} Let $H$ be a separable Hilbert space. Let $\mathds{1}$ be a norm-one element in $H$, and let $p_{\mathds{1}} = \mathds{1} \otimes \overline{\mathds{1}}$ denote the rank-one projection onto $\mathrm{span}\{\mathds{1}\}$. Given $0 < p \leq \infty$, we define the \emph{column $H$-valued $\Lp$ space} as
    \begin{align*}
        \Lp(\mathcal{M};H^c) = \Lp(\mathcal{M} \overline{\otimes} B(H))(\mathbf{1}_{\mathcal{M}} \otimes p_{\mathds{1}})
    \end{align*}
    and the \emph{row $H$-valued $\Lp$ space} as
    \begin{align*}
        \Lp(\mathcal{M};H^{*r}) = (\mathbf{1}_{\mathcal{M}} \otimes p_{\mathds{1}})\Lp(\mathcal{M} \overline{\otimes} B(H)).
    \end{align*}
    Identify $\Lp(\MMM)$ as a subspace of $\Lp(\MMM \overline{\otimes} \BHH)$ via the map $m \mapsto m \otimes p_{\mathds{1}}$. This is equivalent to the identity
    \begin{align*}
        \Lp(\mathcal{M}) = (\mathbf{1}_{\mathcal{M}} \otimes p_{\mathds{1}}) \Lp(\mathcal{M} \overline{\otimes} B(H)) (\mathbf{1}_{\mathcal{M}} \otimes p_{\mathds{1}}).
    \end{align*}
    Thus, given an element $f$ in $\Lp(\mathcal{M};H^c)$, then $f^* f \in L_{p/2}(\MMM)$, which justifies defining
    \begin{align*}
        \|f\|_{\Lp(\mathcal{M};\Hc)} = \|f\|_{\Lp(\mathcal{M}\overline{\otimes} B(H))}=\|(f^*f)^{1/2}\|_{\Lp(\mathcal{M})}.
    \end{align*}
    Analogously, if $f \in \Lp(\MMM;H^{*r})$, then $ff^* \in L_{p/2}(\MMM)$, which enables us to set $\|f\|_{\Lp(\MMM;H^{*r})} = \|f^*\|_{\Lp(\MMM;H^c)}$. We will use without reference that the algebraic tensor $\Lp(\mathcal M)\otimes H$ is dense (resp. weak$^*$ dense) in $\Lp(\mathcal{M};H^c)$ for $1\leq p<\infty$ (resp. $p=\infty$) and similarly for rows.

    Column and row Hilbert-valued $\Lp$-spaces satisfy the expected duality relations expressed via the natural duality bracket 
    \begin{align}
      \la f, g\ra_{r,c} = \mathrm{Tr\otimes \tau }(fg)
    \end{align}
    where $\mathrm{Tr}$ denotes the trace of $B(H)$. More clearly, there holds linearly isometrically
    \begin{align*}
        \Lp(\mathcal{M};\Hc)^* = \Lpprime(\mathcal{M};H^{*r}) \ \mathrm{and} \ \Lp(\mathcal{M};H^{*r})^* = \Lpprime(\mathcal{M};\Hc).
    \end{align*}
    for any $1 \leq p < \infty$ whenever $1/p + 1/p^\prime =1$.

    Let $(\Omega,\mu)$ be a $\sigma$-finite measure space. A remarkable setting for noncommutative Hilbert-valued column/row $\Lp$-spaces is the case $H = \Ltwo(\Omega):= \Ltwo(\Omega,\mu)$. Identifying $\Ltwo(\Omega,\mu)^*$ and $\Ltwo(\Omega,\mu)$ and using the bilinear pairing $(f,g)\mapsto \int_\Omega fg \ d\mu$ yields the following duality identity
    \begin{align*}
        \Lpprime(\mathcal{M};\Ltwor(\Omega,\mu)) = \Lp(\mathcal{M};\Ltwoc(\Omega,\mu))^* \ \mathrm{for \ } 1 \leq p < \infty.
    \end{align*}
    Moreover, for $F=\sum_{i=1}^n m_i\otimes f_i \in \Lp(\mathcal M)\otimes
    L_2(\Omega)$ with $p<\infty$:
    \begin{align}\label{eq:columnLpOmNorm}
        \|F\|^p_{ \Lp(\mathcal{M};\Ltwoc(\Omega,\mu))} = \tau \Big(\int_\Omega \Big|\sum_{i=1}^n f_i(t) m_i\Big|^2 d\mu \Big)^{p/2}  
    \end{align}
    In this work, only the case $p=1$ will be relevant. It turns out that
    $$
    \Lone(\mathcal{M};\Ltwoc(\Omega,\mu)) \subseteq \Ltwo(\Omega;\Lone(\MMM))
    $$
    so that the former space can be identified with a.e. Bochner measurable functions from $\Omega$ to $\Lone(\MMM)$. Moreover, according to the discussion above, $\Lone(\MMM) \otimes \Ltwo(\Omega)$ is dense in $\Lone(\mathcal{M};\Ltwoc(\Omega,\mu))$ with respect to the topology given by \eqref{eq:columnLpOmNorm} for $p=1$.

    \subsection*{Vector-valued molecules} In the classical setting, \emph{molecules} arose as convenient objects to prove that bounded linear operators $T: \Ltwo(\R) \rightarrow \Ltwo(\R)$ admit a continuous extension from $\Hone(\R)$ to itself. This was first noticed by Coifman and Weiss \cite{CW77}, and studied by Meyer and Coifman \cite{MeyC97,Mey90} in the context of Calderón-Zygmund operators. An alternative definition of the Hardy space $\Hone$ via molecules is still present in the context of Bochner measurable functions (see \cite{H06} or \cite[Appendix A]{CR24}). 
    
    For our purposes, it will be enough to consider the case of Bochner measurable functions with values in the Hilbert space $\Ltwo(\MMM)$. The $\Ltwo(\MMM)$-valued Hardy space $\Hone(\R;\Ltwo(\MMM))$ is the subspace of functions $f$ in $\Lone(\R;\Ltwo(\MMM))$ admitting an expression of the form
    \begin{align*}
        f = \sum_{i=0}^{\infty} \lambda_i b_i
    \end{align*}
    where $(\lambda_i)_i \in \ell_1$ and each $b_i$ is an $\Ltwo(\MMM)$-valued atom satisfying the conditions
    \begin{align*}
        \mathrm{supp}_{\R}(b_i) \subseteq I_i, \quad \int b_i = 0, \quad \|b_i\|_{\Ltwo(\R;\Ltwo(\MMM))} \leq \frac{1}{\sqrt{|I_i|}}
    \end{align*}
    for some finite interval $I_i$. Set $\omega(x) = 1+x^2$ and consider the spaces
    \begin{align*}
        \Ltwo(\R,\omega \ dx; \Ltwo(\MMM)) = \big\{ f \in L_0(\R;\Ltwo(\MMM)) \ : \ \int_{\R} \|f(x)\|^2_{\Ltwo(\MMM)} \ \omega(x) \ dx \ < \infty \big\}
    \end{align*}
    and 
    \begin{align*}
        \Ltwo^{\circ}(\R,\omega \ dx; \Ltwo(\MMM)) = \big\{f \in \Ltwo(\R, \omega \ dx ; \Ltwo(\MMM)) \ : \ \int_{\R} f = 0 \big\}.
    \end{align*}
    Then, an \emph{$\Ltwo(\MMM)$-valued molecule} is defined to be a function $f$ in $\Ltwo^{\circ}(\R,\omega \ dx; \Ltwo(\MMM))$ which is normalized by
    \begin{align*}
        \Big( \int_{\R} \|f(x)\|_{\Ltwo(\MMM)}^2 \ \Big(1 + \frac{|x-x_0|^2}{d^2}\Big) \ dx \Big)^{1/2} \leq d^{-1/2}.
    \end{align*}
    Following the argument by Meyer \cite{Mey90}, it can be proved that there is a continuous injection with dense range
    \begin{align}\label{eq:mapQ}
         Q: \Ltwo^\circ( \R, (1+x^2) dx;\Ltwo(\mathcal{M}))  \longrightarrow  \Hone(\R;\Ltwo(\mathcal{M}))
    \end{align}
    which sends each $F$ in $\Ltwo^\circ( \R, (1+x^2) dx;\Ltwo(\mathcal{M}))$ to an atomic decomposition $\sum_{i=0}^\infty \lambda_i b_i$. The invariance by translation and the homogeneity by homotheties of the $\Hone(\R;\Ltwo(\MMM))$-norm implies that the norm in this space of any molecule is bounded by a universal constant. Ultimately, this implies that one can use molecules instead of atoms in the definition of $\Hone(\R;\Ltwo(\MMM))$.

    \section{$c$-molecules and proof of Theorem~\ref{thm:CZH1c}}

    
    It seems that the proof of the main result of this paper should follow the same scheme as in the classical setting \cite{MeyC97}. That is, one should try to prove that $T$ sends $c$-atoms to a column version of molecules. This is made possible due to a \textit{partial} link between the vector-valued and the semicommutative theory. Indeed, it was proved in \cite{CR24} that the operator $Q$ from \eqref{eq:mapQ} extends to a bounded injective map with dense range
    \begin{align}\label{eq:tildeQ}
         \widetilde{Q}: \Lone(\mathcal{M};\Ltwoczero(\mathbb{R},(1+x^2) dx))  \longrightarrow  \Hardyc
    \end{align}
    defined by the identity
    \begin{align*}
        \widetilde{Q}(Fh):= Q(F)h
    \end{align*}
    for any $h \in \Ltwo(\mathcal{M})$ and $F \in \Ltwo^\circ(\mathbb{R},(1+x^2)dx) \otimes \Ltwo(\mathcal{M})$. Recall that the range of $\widetilde{Q}$ is actually dense in $\Hardyc$ since it contains all the $c$-atoms of the form $a=bh$ with $b \in \Ltwo(\mathbb{R}) \otimes \Ltwo(\mathcal{M})$.


    \begin{definition}
        A $c$-molecule $f$ in $\Hardyc$, centered at $x_0$ and of width $d > 0$, is defined to be a function such that $f = \widetilde{Q}(F)$ for some $F$ in $\Lone(\mathcal{M};\Ltwoczero(\mathbb{R},(1+x^2) dx))$ satisfying
        \begin{align*}
            \|d \cdot F(d x +x_0)\|_{\Lone(\mathcal{M};\Ltwoczero(\mathbb{R},(1+x^2)dx))} \leq 1,
        \end{align*}        
        or, equivalently,
        \begin{align*}
            \|F\|_{\Lone(\mathcal{M};\Ltwoczero(\mathbb{R},(1+\frac{|x-x_0|^2}{d^2}) dx))} \leq d^{-1/2}.
        \end{align*}
    \end{definition}
    

    Given a $c$-molecule centered at $x_0$ and of width $d$,
    \begin{align*}
        \|f\|_{\Hardyc} = \|d \cdot f(dx+x_0)\|_{\Hardyc} \lesssim \|d \cdot F(dx+x_0)\|_{\Lone(\mathcal{M};\Ltwoczero(\mathbb{R},(1+x^2)dx))},
    \end{align*}
    so there follows that the $\Honec$-norm of any $c$-molecule is bounded by a universal constant. Therefore, $c$-atoms can be replaced by $c$-molecules in the definition of the $\Hardyc$-norm, yielding
    \begin{align}\label{eq:H1cMoleculeNorm}
        \|f\|_{\Hardyc} \simeq \inf\Big\{ \sum_{i=0}^\infty |\lambda_i| \ : f = \sum_{i=0}^\infty \lambda_i f_i \ \mbox{in} \ \Lone(\mathcal{A}), \ (\lambda_i)_i \in \ell_1, \ (f_i)_i \ c\mbox{-molecules}\Big\}.
    \end{align}

    As mentioned above, Theorem~\ref{thm:CZH1c} relies on the connection between the maps $Q$ and $\widetilde{Q}$, which allows us to reduce the problem to studying the boundedness of Calderón-Zygmund operators with operator-valued kernel from $\Hone(\mathbb{R};\Ltwo(\mathcal{M}))$ to itself.
    

    \begin{lemma}\label{lem:vectorMolecule}
        Let $\mathcal{M}$ be a von Neumann algebra. Let $T$ be a left Calderón-Zygmund operator which is bounded on $\Ltwo(\mathbb{R};\Ltwo(\mathcal{M}))$ and has associated kernel
        \begin{align*}
            K : \mathbb{R} \times \mathbb{R} \ \setminus \ \{x = y \} \longrightarrow \mathcal{M}.
        \end{align*}
        Assume that $K$ satisfies the Lipschitz condition \eqref{eq:smoothnessCondition}. Then, $T$ extends to a map from $\Hone(\mathbb{R};\Ltwo(\mathcal{M}))$ to itself if and only if $\int T(b) = 0$ for every $\Ltwo(\MMM)$-valued atom $b$.
    \end{lemma}
    \begin{proof}    
        The approximation argument for singular kernels which was introduced in \cite{CR24} can be adapted to prove that the Calderón-Zygmund operator $T$ is well-defined on the whole $\Hone(\R,\Ltwo(\MMM))$. Let $b$ be a $\Ltwo(\MMM)$-valued atom. Assume that $\mathrm{supp}_{\R}(b) \subset I$ for some interval $I$ centered at $x_0$ with radius $d$, and let $\lambda I$ be the interval centered at $x_0$ with radius $\lambda d$. Then, 
        \begin{align*}
            \Bigg( \int_{\R} \|T(b)\|_{\Ltwo(\mathcal{M})}^2 \Big( 1 + \frac{|x-x_0|^2}{d^2} \Big) \ dx \Bigg)^{1/2} &\lesssim \Bigg( \int_{\lambda I} \|T(b)\|_{\Ltwo(\mathcal{M})}^2 \ dx \Bigg)^{1/2} \\
            &+ \Bigg( \int_{(\lambda I)^c} \|T(b)\|_{\Ltwo(\mathcal{M})}^2 \frac{|x-x_0|^2}{d^2} \ dx \Bigg)^{1/2}.
        \end{align*}
        The first integral can be bounded as a consequence of the boundedness of $T$ on $\Ltwo(\mathcal{A})$, that is,
        \begin{align*}
            \Bigg( \int_{\lambda I} \|T(b)\|_{\Ltwo(\mathcal{M})}^2 \ dx \Bigg)^{1/2}
            &\leq \|T(b)\|_{\Ltwo(\mathcal{A})} \leq \|T\| \|b\|_{\Ltwo(\mathcal{A})} \leq \frac{\|T\|}{2^{1/2}} \ d^{-1/2}
        \end{align*}
        On the other hand,
        \begin{align*}
            \Big\|T(b) \bigchi_{(\lambda I)^c} \frac{|x-x_0|}{d} \Big\|_{\Ltwo(\R;\Ltwo(\mathcal{M}))} = \sup_g \Big| \tau \int T(b) \ \frac{|x-x_0|}{d} \  g \Big|
        \end{align*}
        where the supremum is taken over $g \in \mathcal{S}$ supported on $(\lambda I)^c$ such that $\|g\|_{\Ltwo((\lambda I)^c;\Ltwo(\mathcal{M}))} \leq 1$. Assume for the moment that $b \in \mathcal{S}$. Then, there holds
        \begin{align*}
            \Big\|T(b) &\bigchi_{(\lambda I)^c} \ \frac{|x-x_0|}{d}\Big\|_{\Ltwo(\R,\Ltwo(\mathcal{M}))} \\
            &= \sup_g \Big| \tau \int_{(\lambda I)^c} \int_{I} K_{(\lambda I)^c,I}(x,y) \ b(y) \ \frac{|x-x_0|}{d} \ g(x) \ dx \ dy \Big|.
        \end{align*}
        The $c$-atom $b$ having integral zero enables us to write
        \begin{align*}
            \Big\|T(b) &\bigchi_{(\lambda I)^c} \ \frac{|x-x_0|}{d}\Big\|_{\Ltwo(\R,\Ltwo(\mathcal{M}))}\\  &= \sup_g \Big| \tau \int_{(\lambda I)^c} \int_{I} (K_{(\lambda I)^c,I}(x,y) - K_{(\lambda I)^c,I}(x,x_0)) \ b(y) \ \frac{|x-x_0|}{d} \ g(x) \ dx \ dy \Big| \\
            &= \Big\| \frac{|x-x_0|}{d} \int_I (K_{(\lambda I)^c,I}(x,y) - K_{(\lambda I)^c,I}(x,x_0)) \ b(y) \ dy\Big\|_{\Ltwo((\lambda I)^c;\Ltwo(\mathcal{M}))} \\
            &= \Bigg( \int_{(\lambda I)^c} \Big\| \int_{I}  (K_{(\lambda I)^c,I}(x,y) - K_{(\lambda I)^c,I}(x,x_0)) \ b(y) \ dy \Big\|_{\Ltwo(\mathcal{M})}^2 \ \ \frac{|x-x_0|^2}{d^2} \ \ dx \Bigg)^{1/2}.
        \end{align*}
        The Lipschitz condition \eqref{eq:smoothnessCondition} implies that for a.e. $x$,
        \begin{align*}
            \Big\| \int_{I}  (K_{(\lambda I)^c,I}(x,y) &- K_{(\lambda I)^c,I}(x,x_0)) \ b(y) \ dy \Big\|_{\Ltwo(\mathcal{M})} \\ &\leq  \int_{I} \| (K_{(\lambda I)^c,I}(x,y) - K_{(\lambda I)^c,I}(x,x_0)) \ b(y) \|_{\Ltwo(\mathcal{M})} \ dy \\
            &\leq \int_I \|K_{(\lambda I)^c,I}(x,y) - K_{(\lambda I)^c,I}(x,x_0)\|_{\mathcal{M}} \ \|b(y)\|_{\Ltwo(\mathcal{M})} \ dy \\
            &\lesssim \int_I \frac{|x_0 - y|^\gamma}{|x-x_0|^{1+\gamma}} \ \|b(y)\|_{\Ltwo(\mathcal{M})} \ dy \\
            &\leq |x-x_0|^{-1-\gamma} \ d^\gamma.
        \end{align*}
        Therefore, 
        \begin{align*}
            \Bigg( \int_{(\lambda I)^c} \|T(b)\|_{\Ltwo(\mathcal{M})}^2 &\frac{|x-x_0|^2}{d^2} \ dx \Bigg)^{1/2} \leq \Bigg( \int_{(\lambda I)^c} |x-x_0|^{-2\gamma} d^{2\gamma -2} \ dx \Bigg)^{1/2} \\
            &= \Bigg( \int_{(-\lambda,\lambda)^c} |x|^{-2\gamma} \ dx \Bigg)^{1/2} \ d^{-1/2} = C_{\lambda,\gamma} \ d^{-1/2}.
        \end{align*}
        Notice that $\gamma \in (1/2,1]$ implies that the constant $C_{\lambda,\gamma}$ is finite. The same bound holds for any $b \in \Ltwo(\mathcal{A})$ by a standard approximation argument (see \cite[Lemma 3.2]{CR24}). Finally, it is clear that $\int T(b) = 0$ follows by hypothesis, so $T(b)$ is proven to be a $\Ltwo(\mathbb{R};\Ltwo(\mathcal{M}))$-molecule. Reciprocally, if $T$ is bounded from $\Hone(\mathbb{R};\Ltwo(\mathcal{M}))$ to itself, then $T(b)$ must have zero integral.
    \end{proof}

    Now, we are ready to prove the main result of this paper.

    \begin{proof}[Proof of Theorem \ref{thm:CZH1c}]    
        Let $a=bh$ be a $c$-atom. Then, $T(a) = T(b)h$ holds in $\Lone(\mathcal{A})$ and is well-defined without regard to the decomposition $a=bh$ for $a$ \cite[Theorem 3.5]{CR24}. It is clear that $\int_{\mathbb{R}} T(b) = 0$ holds, and Lemma~\ref{lem:vectorMolecule} implies that there exists a $\Ltwo(\mathcal{M})$-valued molecule $F$ centered at $x_0$ with width $d$ such that $Q(F)h=T(b)h$. Moreover, whenever $F$ is in $\Ltwozero(\R,(1+x^2) dx) \otimes \Ltwo(\mathcal{M})$, then $\widetilde{Q}(F h) = T(b)h$ and 
        \begin{align*}
            \|\widetilde{Q}(Fh)\|_{\Lone(\mathcal{M};\Ltwoczero(\mathbb{R},(1+\frac{|x-x_0|^2}{d^2}) dx))} \leq \|F\|_{\Ltwo(\mathcal{M};\Ltwoczero(\mathbb{R},(1+\frac{|x-x_0|^2}{d^2}) dx))} \|h\|_{\Ltwo(\mathcal{M})} \leq d^{-1/2},
        \end{align*}
        so $T(b)h$ is a $c$-molecule. However, this does not happen in general, but proving that the expression $\|T(a)\|_{\Honec(\mathcal{A})}$ is bounded by a universal constant for any $c$-atom is enough. Indeed, Lemma~\ref{lem:vectorMolecule} yields that 
        \begin{align*}
            \|T(b)h\|_{\Hardyc} \leq \|T(b)\|_{\Hone(\R;\Ltwo(\mathcal{M}))} \lesssim 1.
        \end{align*}
        Therefore, the equivalence of norms
        \begin{align*}
            \|f\|_{\Hardyc} \simeq \inf\Big\{ \sum_{i=0}^{\infty} |\lambda_i| \ : \ f = \sum_{i=0}^\infty \lambda_i f_i \ , \ (\lambda_i)_i \in \ell_1, \ \|f_i\|_{\Hardyc} \leq 1 \Big\}.
        \end{align*}
        implies the statement of the theorem.
    \end{proof}

    \begin{remark}
        The map $Q$ in \eqref{eq:mapQ} induces a bilinear form which yields the extension map
        \begin{align*}
            \breve{Q} : \Ltwo^\circ( \R, (1+t^2) dt;\Ltwo(\mathcal{M})) \ \widehat{\otimes}_{\pi} \ \Ltwo(\mathcal{M}) \longrightarrow \Honec(\mathcal{A})
        \end{align*}
        which satisfies $\breve{Q}(F \otimes h) = Q(F)h$ for every $F \in \Ltwo^\circ( \R, (1+t^2) dt;\Ltwo(\mathcal{M}))$ and $h \in \Ltwo(\MMM)$. The map $\breve{Q}$ can be proved to be a bounded linear operator with dense range (see \cite{CR24}), and provides an alternative definition of $c$-molecules such that any left Calderón-Zygmund operator sends $c$-atoms to $c$-molecules. We have chosen to use the map $\widetilde{Q}$ as in \eqref{eq:tildeQ} instead of $\breve{Q}$ since the former was used to prove the $\Hone$-$\BMO$ duality in \cite{CR24}, although $\breve{Q}$ can be shown to provide such a result as well. 
    \end{remark}

    \bibliographystyle{amsplain}
    \bibliography{biblio}

\end{document}